\pdfoutput=1
\documentclass[11pt,a4paper]{amsart}
\usepackage[T1]{fontenc}
\usepackage[margin=1.1in]{geometry}
\usepackage{graphicx}

\usepackage{amsthm, amsmath, amsfonts, amssymb, mathtools}
\usepackage{bbold}
\usepackage{booktabs}
\usepackage[dvipsnames]{xcolor}
\usepackage{tikz}
\usepackage{subcaption}
\usepackage{algorithm}
\usepackage[noend]{algpseudocode}
\usepackage{longtable}
\usepackage{makecell}
\usepackage{bbm}
\usepackage{pgfplots}
\pgfplotsset{compat=1.18}
\usepackage{pgffor}

\usepackage{xspace}
\usepackage{csquotes}
\usepackage[colorlinks=true, pdfstartview=FitV, linkcolor=Blue, citecolor=OliveGreen, urlcolor=BrickRed, breaklinks=false]{hyperref}
\hypersetup{hypertexnames=false}

\usetikzlibrary{arrows,arrows.meta,decorations.pathreplacing,shapes}
\usetikzlibrary{positioning,fit}
\usetikzlibrary{graphs,graphs.standard}
\usetikzlibrary{calc}

\usepackage[colorinlistoftodos, bordercolor=orange, backgroundcolor=orange!20, linecolor=orange, textsize=scriptsize]{todonotes}

\newcommand\TOPCOM{\texttt{TOPCOM}\xspace}

\newcommand{\RR}{\mathbb{R}}
\newcommand{\ZZ}{\mathbb{Z}}
\DeclareMathOperator\conv{conv} %
\newcommand\Sym[1]{\mathfrak{S}_{#1}} %

\newcommand\subgroup[1][G]{\mathfrak{#1}}
\newcommand\symmetry{\pi}
\newcommand\affineSymmetry{\symmetry}
\newcommand\pointconfig{\mathbf{A}}

\newcommand\bigO{\mathcal{O}}

\newcommand\dTriang[1][d]{#1\cdot\Delta_2}
\newcommand\dTpoints[1][d]{\dTriang \cap \ZZ^2}
\newcommand\dHalfTriang[1][d]{#1\cdot T}
\newcommand\subdivision{\Sigma}
\newcommand\triang{\sigma}
\newcommand\anotherTriang{\tau}
\newcommand\triangulation{\subdivision}
\newcommand\admissibles[2]{\operatorname{adm}_{#1}(#2)}
\newcommand\localadm{\mathsf{adm}}
\newcommand\lexmin{\operatorname{lex-min}}

\newcommand{\symmetricCapacity}{\widetilde{c}}

\newcommand{\numUnimodular}{F_\Delta}
\newcommand{\logUnimodular}{f_\Delta}

\newcommand{\numHalf}{F_T}
\newcommand{\logHalf}{f_T}

\newcommand{\numRectangle}{F_\square}

\newcommand{\symmetryAxis}{\{x=y\}}

\newcommand{\numSymmetricUnimodular}{\widetilde{F}_\Delta}
\newcommand{\logSymmetricUnimodular}{\widetilde{f}_\Delta}

\newcommand{\numInteriorEdges}{e}

\newcommand{\lowerBound}[1]{\ell_{#1}}
\newcommand{\LowerBound}[1]{L_{#1}}
\newcommand{\upperBound}{u}
\newcommand{\UpperBound}{U}

\newcommand\cprime{$'$} %

\theoremstyle{plain}
\newtheorem{theorem}{Theorem}
\newtheorem{lemma}[theorem]{Lemma}
\theoremstyle{definition}

\newtheorem{remark}[theorem]{Remark}
\newtheorem{conjecture}[theorem]{Conjecture}
\newtheorem*{notation}{Notation}

\newcommand\scalefactor{0.8} %

\usepackage[backend=biber,style=alphabetic,doi=false,isbn=false,url=false,giveninits=true,maxbibnames=50]{biblatex}
\title{Counting symmetric unimodular triangulations}
\author{Kamillo Ferry}
\address[Kamillo Ferry]{
  Imperial College London,
  Department of Mathematics, London SW7 2AZ, United Kingdom}
\email{k.ferry@imperial.ac.uk}
\author{Michael Joswig}
\address[Michael Joswig]{
	Technische Universität Berlin,
	Chair of Discrete Mathe\-ma\-tics/Geo\-me\-try, Str. des 17. Juni 136, 10587 Berlin, Germany \\
	and Max-Planck Institute for Mathematics in the Sciences, Leipzig, Germany}
\email{joswig@math.tu-berlin.de}
\author{J\"org Rambau}
\address[J\"org Rambau]{
  University of Bayreuth\\
  Chair for Economathematics,
  95440 Bayreuth,
  Germany}
\email{Joerg.Rambau@uni-bayreuth.de}


\begin{document}

\begin{abstract}
  The objects of study are triangulations of the dilated standard triangle in the plane.
  Motivated by work on T-curves (Geiselmann et al., 2026), the focus lies on unimodular triangulations with a fixed symmetry axis.
  Lower and upper bounds are given, in combination with full enumerations of a few small cases.
\end{abstract}

\maketitle

\section{Introduction}

There are many reasons to study triangulations of finite point configurations; for an overview, the reader is referred to \cite[Chapter 1]{LoeraRambauSantos_TriangulationsStructuresApplications_2010}.
Here our perspective is influenced by the combinatorial view on algebraic geometry pioneered by Gel{\cprime}fand, Kapranov and Zelevinsky \cite{GKZ}.
The latter line of research is nowadays continued in tropical geometry \cite{Tropical+Book,ETC}.
In recent work Geiselmann et al.\ gave an explicit description as T-curves of all real schemes of plane projective algebraic curves of degrees $d\leq 7$ \cite{GeiselmannJoswigKastnerMundingerPokuttaSpiegelWackZimmer:2602.06888}. 
The real schemes of degree seven have been classified by Viro \cite{Viro:1984}.
A \emph{T-curve} of degree $d$ is determined by a unimodular triangulation of the dilated standard triangle $\dTriang=\conv\{(0,0),(d,0),(0,d)\}$ and a sign distribution on the $\tfrac{1}{2}(d+2)(d+1)$ lattice points of $\dTriang$; see also \cite[pp.~28ff]{LoeraRambauSantos_TriangulationsStructuresApplications_2010}.
A triangulation of a lattice polygon, such as $\dTriang$, is unimodular if and only if it uses all the lattice points in the polygon.
It turns out that surprisingly few triangulations suffice to produce all real schemes in all degrees $d\leq 7$.
Moreover, it even suffices to consider triangulations of $\dTriang$ which are symmetric with respect to the line $\symmetryAxis$. 
This leads us to investigating such triangulations more closely, addressing \cite[Question 27]{GeiselmannJoswigKastnerMundingerPokuttaSpiegelWackZimmer:2602.06888}.
Usually, in the context of algebraic curves, the definition of a T-curve requires the triangulation to be regular (induced by a height function; cf.\ \cite[Chapter 5]{LoeraRambauSantos_TriangulationsStructuresApplications_2010}).
The significance of nonregular triangulations for the connection to algebraic geometry is subject to active research; see, e.g., recent work of Brugall\'e, L\'opez de Medrano and Rau \cite{2024:BrugalleEtAl:CombinatorialPatchworking}.

The set of all triangulations of a fixed finite point set in the plane has a special structure: it is connected by local modifications known as \emph{flips}; cf.\ \cite[\S3.4.1]{LoeraRambauSantos_TriangulationsStructuresApplications_2010}.
This property can be exploited for counting such triangulations, though this is not necessarily the fastest option as of today \cite{Rambau_Symmetriclexicographicsymmetric_2026}.
General results on the number of planar triangulations as a function of the number of points have been given by Sharir and Sheffer \cite{Sharir+Sheffer:2011} (upper bound) and Aichholzer, Hurtado and Noy \cite{Aichholzer+Hurtado+Noy:2004} (lower bound); see also \cite[\S3.2.3 \& 3.2.4]{LoeraRambauSantos_TriangulationsStructuresApplications_2010}.
Kaibel and Ziegler counted unimodular triangulations of lattice rectangles \cite{KaibelZiegler:2003}; their work was recently extended by Orevkov \cite{Orevkov:2022}.

Here we are interested in the following quantities and their asymptotics:
\begin{align*}
  \numUnimodular(d) &\coloneqq \#\{\text{unimodular triangulations of } \dTriang\},\text{ and}\\
  \numSymmetricUnimodular(d) &\coloneqq \#\{\text{symmetric unimodular triangulations of } \dTriang\},
\end{align*}
where ``symmetric'' refers to the axial symmetry with respect to the
line $\symmetryAxis$.  For \(d\leq 9\) we were able to obtain exact
numbers for \(\numSymmetricUnimodular(d)\) using a method from
\cite{Rambau_Symmetriclexicographicsymmetric_2026} while we also give
asymptotic lower and upper bounds for general \(d\).  As an add-on we
present the numbers of regular symmetric unimodular triangulations
$\numSymmetricUnimodular(d)_{\mathrm{regular}}$ for $d \le 7$.
Enumerating or counting triangulations is known to be hard
(theoretically and in practice).  Hence, tight asymptotic results are
the best that one can hope for.  Table~\ref{tbl:the-numbers} below
summarizes the results.

\begin{notation}
Throughout we employ the convention that counting functions are denoted by uppercase letters, e.g., $\numUnimodular$, $\LowerBound1$ and so on,
whereas lowercase letters denote the base-2 logarithm of the respective functions, e.g.,\ $\logUnimodular(d) = \log\numUnimodular(d)$.
\end{notation}

\section{Exact enumeration}\label{sec:exact-enumeration}

In \cite{Rambau_Symmetriclexicographicsymmetric_2026} a new method to
enumerate symmetric subsets up to symmetry was developed and applied
to triangulations.  Formulated in the language of triangulations, the
general method needs the following notions.

Let $\pointconfig$ be a \emph{point configuration} $\pointconfig$ in $\RR^D$ that is affinely spanning.
If the cardinality of $\pointconfig$ equals $n$, the \emph{affine symmetries} form a subgroup $\subgroup$ of $\Sym{n}$ containing all permutations of the points which are induced by affine transformations of $\RR^D$.
For another subgroup $\subgroup[H]$ of $\Sym{n}$, a triangulation
$\triangulation = \{ \triang_1, \triang_2, \dots, \triang_k\} \subseteq \binom{[n]}{D+1}$
of~$\pointconfig$ with $D$-dimensional simplices $\triang_j \in \binom{[n]}{D+1}$,
$j \in [k]$, is \emph{$\subgroup[H]$-invariant}, if
\[
  \affineSymmetry(\triangulation) = \{ \affineSymmetry(\triang_1), \affineSymmetry(\triang_2), \dots, \affineSymmetry(\triang_k) \} = \triangulation
\]
for all $\affineSymmetry \in \subgroup[H]$, where $\affineSymmetry(\triang=\{\triang_1, \triang_2, \dots, \triang_r\}) = \{\affineSymmetry(\triang_1), \affineSymmetry(\triang_2), \dots, \affineSymmetry(\triang_r)\}$.  

A simplex $\triang$ is \emph{$\subgroup[H]$-feasible} if its $\subgroup[H]$-orbit is pairwise properly intersecting, i.e., if its whole orbit under $\subgroup[H]$ could coexist in a triangulation of~$\pointconfig$.
The stabilizer $\subgroup_{\subgroup[H]}$ of the $\subgroup[H]$-feasible simplices is called \emph{the $\subgroup[H]$-feasible symmetry group}.
A pair $(\triang, \anotherTriang)$ of simplices is called \emph{$\subgroup[H]$-admissible}, if the $\subgroup[H]$-orbits of $\triang$ and~$\anotherTriang$ form a pairwise properly intersecting set of simplices.
A simplex $\triang$ is $\subgroup[H]$-admissible w.r.t.\ a set of simplices $\triangulation$, if $\triang$ and~$\anotherTriang$ are an $\subgroup[H]$-admissible pair for all $\anotherTriang \in \triangulation$.
Let $\admissibles{\subgroup[H]}{\triangulation}$ be the set of simplices that are $\subgroup[H]$-admissible w.r.t.\ $\triangulation$ and that are lexicographically larger (lex-larger, in brief) than the lex-maximal simplex in~$\triangulation$.

For the special task in this paper, $\pointconfig = \dTpoints$,
$\subgroup$ is its symmetry group, and $\subgroup[H]$ is the
subgroup generated by the reflection \(\affineSymmetry\) at the line
$\symmetryAxis \subset \RR^2$.  The exact enumeration of all
symmetric triangulations of $\pointconfig$ up to symmetry now corresponds
to the enumeration of all $\subgroup[H]$-invariant triangulations
of~$\pointconfig$ up to $\subgroup[H]$-feasible symmetries.

The \emph{symmetric lexicographic symmetric-subset reverse search}
method proposed in~\cite{Rambau_Symmetriclexicographicsymmetric_2026}
does exactly this by extending partial triangulations by one simplex
at a time in lex-order (\emph{lex-extending}).  It is reminiscent of
\emph{orderly generation} and of \emph{reverse search} from other
contexts; cf.\ \cite{Avis+Fukuda:1996}.  More specifically, the
algorithm can be described by the following recursive function, which
in the beginning is called for the empty set.

\begin{algorithm}[ht]
  \begin{algorithmic}[1]
    \If{$\triangulation$ is lex-min in its $\subgroup_{\subgroup[H]}$-orbit}
    \If{$\triangulation$ is a full triangulation}
    \State output $\triangulation$ 
    \Else{}
    \State $\localadm \gets \admissibles{\subgroup[H]}{\triangulation}$
    \While{$\localadm$ is non-empty}
    \State $\triangulation' \gets \triangulation \cup \lexmin (\localadm)$.
    \State outputAllTriangsLexContaining($\triangulation'$)
    \State $\localadm \gets \localadm \setminus \lexmin (\localadm)$
    \EndWhile
    \EndIf
    \EndIf
  \end{algorithmic}
  \caption{Function outputAllTriangsLexContaining($\triangulation$)}
\end{algorithm}
\begin{theorem}[\cite{Rambau_Symmetriclexicographicsymmetric_2026}]
  A call of outputAllTriangsLexContaining($\varnothing$) outputs each
  $\subgroup[H]$-in\-var\-iant triangulation exactly once up to
  $\subgroup[H]$-feasible symmetries. \qed
\end{theorem}
In its basic form, the algorithm would enumerate all maximal subsets
of pairwise $\subgroup[H]$-admissible simplices. In general
dimensions, this is inefficient because it usually produces a lot of
deadends.  The reason is that not all partial triangulations can be
lex-extended to a triangulation. One might think that in dimension two
such deadends are impossible, since in dimension two all partial
triangulations can be extended to a triangulation.  However, even in
dimension one and two not all partial triangulations~$\triangulation$
can be \emph{lex}-extended, i.e., extended to a triangulation using
only simplices that are \emph{lex-larger} than the ones already
in~$\triangulation$.  This is recursively relevant in all branches
other than the first one. Therefore, sophisticated pruning methods
have been developed
in~\cite{Rambau_Symmetriclexicographicsymmetric_2026}.  The most
efficient known pruning method can be sketched as follows: If the
lex-minimal uncovered interior facet of a simplex in a incomplete
triangulation~$\triangulation$ is lex-smaller than the lex-minimal
facet of simplex $\subgroup[H]$-admissible w.r.t.~$\triangulation$,
then $\triangulation$ cannot be lex-extended to a triangulation in any
future branch.

With the resulting algorithm as implemented in \TOPCOM
\cite{Rambau_TOPCOM-zenodo_2026} (together with specialized methods to check
for lex-minimality in orbits and specialized update methods for data
structures), the exact numbers of symmetric triangulations of
$\dTpoints$ could be enumerated up to $d = 7$.  If the triangulations
are further restricted to unimodular triangulations, then the numbers
could be computed up to $d = 9$ (see Table~\ref{tbl:the-numbers}).
The numbers of regular symmetric unimodular triangulations for
$d \le 7$ have been added for convenience.  The data and the relevant
\TOPCOM instructions are available on Zenodo,
doi:10.5281/zenodo.23213763.

\begin{table}[htbp]
  \centering
  \caption{Number of \(\subgroup[H]\)-invariant triangulations of $\dTriang$ up to \(\subgroup[H]\)-feasible symmetries
    in comparison with some lower and upper bounds. Here \(\numHalf(d)\) is defined in \eqref{eq:def-half}. The upper bound $\UpperBound(d)$ from (\ref{eq:upper}) is omitted due to space constraints.}
  \label{tbl:the-numbers}
  \footnotesize
  \begin{tabular}{*{6}r}
    \toprule
    $d$ &   $\LowerBound2(d)$ &  $\numHalf(d)$ &   $\numSymmetricUnimodular(d)$ & $\numSymmetricUnimodular(d)_{\mathrm{regular}}$ & $2^{\lfloor\frac{d}2\rfloor}\cdot\numHalf(d)$ \\
    \midrule
    1 &   1 &                  1 &               1 &                1 &          1 \\   
    2 &   1 &                  1 &               2 &                2 &          2 \\                                           
    3 &   2 &                  4 &               7 &                7 &          8 \\                                           
    4 &   9 &                  24 &              74 &               74 &         96 \\                                           
    5 &   54 &                 446 &             1,194 &            1,194 &      1,784 \\                                        
    6 &   729 &                14,057 &          63,024 &           62,960 &     112,456 \\                                      
    7 &   14,580 &             1,214,208 &       4,739,031 &        4,729,700 &  9,713,664 \\                                    
    8 &   613,089 &            189,222,465 &     1,211,875,888 &    ? &          3,027,559,440 \\                                
    9 &   42,916,230&          75,358,380,679 &  422,664,577,207 &  ? &          1,205,734,090,864 \\                            
    \bottomrule
  \end{tabular}
\end{table}

\section{Lower bounds}\label{sec:lower-bounds}

A first lower bound for the number of \(\subgroup[H]\)-invariant unimodular triangulations up to
\(\subgroup[H]\)-feasible symmetries can be obtained using the following approach.
Let \(\dHalfTriang\) be one half of \(\dTriang\), that is, the triangle given by \[
  \dHalfTriang \coloneqq \conv\Bigl\{
    (0,0),\Bigl(\frac{d}2,\frac{d}2\Bigr),(d,0)
  \Bigr\}.
\]
For \(d\) even, this is a lattice triangle itself; geometrically it is a rotated and scaled copy of \(\dTriang[\frac{d}2]\), though only with respect to the sublattice \(\{(a,b)\in\mathbb{Z}^2 \mid a+b \text{ even}\}\).
On the other hand, for odd \(d\), \(\dHalfTriang\) is a lattice polygon except for the vertex \((\frac{d}2,\frac{d}2)\).
In this case, we consider unimodular triangulations of the lattice trapezoid \[
  \dHalfTriang' = \conv\Bigl\{
    (0,0),
    \Bigl(\Bigl\lfloor\frac{d}2\Bigr\rfloor,\Bigl\lfloor\frac{d}2\Bigr\rfloor\Bigr),
    \Bigl(\Bigl\lceil\frac{d}2\Bigr\rceil,\Bigl\lfloor\frac{d}2\Bigr\rfloor\Bigr),
    (d,0)
  \Bigr\}
\] and extend it to a triangulation of \(\dHalfTriang\) by adding \((\frac{d}2,\frac{d}2)\) and the resulting triangle Figure~\ref{fig:odd-small-triangle}.
By slight abuse of terminology, we refer to these as unimodular triangulations of \(\dHalfTriang\) regardless of \(d\) even or odd.
This is a reasonable choice since we consider \(\subgroup[H]\)-invariant unimodular triangulations up to symmetries.
Accordingly, we define
\begin{equation}\label{eq:def-half}
  \numHalf(d) \coloneqq \begin{cases}
    \#\{\text{unimodular triangulations of } \dHalfTriang\},& d\text{ even},\\
    \#\{\text{unimodular triangulations of } \dHalfTriang'\},& d\text{ odd}.
  \end{cases}
\end{equation}
For \(d=1\), the trapezoid \(\dHalfTriang'\) degenerates to a segment and we set \(\numHalf(1) \coloneqq 1\).

\begin{figure}[th]
\begin{tikzpicture}[
  scale=\scalefactor,
  line width=1pt,
  roundnode/.style={circle, draw=black, fill=black, thin, inner sep=1pt},
  region/.style={draw=blue,fill=blue,fill opacity=.25,line width=2pt,line cap=round,line join=round},
  tip/.style={draw=orange,fill=orange,fill opacity=.25,line width=2pt,line cap=round,line join=round},
  mirror/.style={dotted,opacity=.25,fill opacity=.15,line width=1pt}
  ]
  \begin{axis}[
    axis equal image,
    grid=both,
    ymin=0,ymax=5.5,xmin=0,xmax=5.5,
    xtick distance=1,
    ytick distance=1,
    xticklabel={\color{gray}\pgfmathprintnumber\tick},
    yticklabel={\color{gray}\pgfmathprintnumber\tick},
    axis x line=middle,
    axis y line=middle,
    ]
    \draw[black,line width=1pt] (axis cs:5,0) -- (axis cs:0,5);
    \draw[blue,line width=2pt,dotted] (axis cs:0,0) -- (axis cs:5.5,5.5);

    \draw[region] (axis cs:0,0) -- (axis cs:5,0) -- (axis cs:3,2) -- (axis cs:2,2) -- cycle;
    \draw[tip] (axis cs:2,2) -- (axis cs:3,2) -- (axis cs:2.5,2.5) -- cycle;
    \draw[tip,mirror] (axis cs:2,2) -- (axis cs:2,3) -- (axis cs:2.5,2.5) -- cycle;
  \end{axis}
\end{tikzpicture}
  \caption{The half-triangle \(\dHalfTriang\) for odd \(d\). The lattice polygon \(\dHalfTriang'\) (blue) and the remaining triangle with
  vertex \((\frac{d}2,\frac{d}2)\) (orange), together with its mirror image, forms a unimodular triangle of \(\dTriang\).}
  \label{fig:odd-small-triangle}
\end{figure}

We can choose any unimodular triangulation of $\dHalfTriang$ and flip it along the \(\symmetryAxis\) line to obtain a triangulation of \(\dTriang\).
In particular, this yields a \(\subgroup[H]\)-invariant triangulation.
For even \(d\), a triangle \(\sigma\subset\dHalfTriang\) is already \(\subgroup[H]\)-feasible, 
and unimodularity is also preserved under this operation.
For odd \(d\), this also holds for triangles \(\sigma\subset\dHalfTriang'\). Then, the additional triangle \(\sigma\subset\dHalfTriang\setminus\dHalfTriang'\)
gets completed into a unimodular triangle. 

An example for even \(d\) is shown in Figure~\ref{fig:flip-small-triangle}.
Since every unimodular triangulation of \(\dHalfTriang\) yields a distinct symmetric unimodular triangulation of \(\dTriang\) in this way,
for all \(d\) we obtain the lower bound
\begin{equation}\label{eq:half-lower-bound}
  \numHalf(d) \leq \numSymmetricUnimodular(d).
\end{equation}

A first-order approximation for the number of unimodular triangulations of \(\dHalfTriang\) can be obtained using a similar argument as 
in \cite{KaibelZiegler:2003}. We subdivide \(\dHalfTriang\) into the triangles below the two legs and vertical strips of width one below,
e.g., as in Figure~\ref{fig:small-triangle-strips}.

\begin{figure}[hb]
\begin{tikzpicture}[
  scale=\scalefactor,
  line width=1pt,
  roundnode/.style={circle, draw=black, fill=black, thin, inner sep=1pt},
  triang/.style={draw=blue,fill=blue,fill opacity=.25,line width=2pt,line cap=round,line join=round},
  mirror/.style={dotted,opacity=.25,fill opacity=.15,line width=1pt}
  ]
  \begin{axis}[
    axis equal image,
    grid=both,
    ymin=0,ymax=6.5,xmin=0,xmax=6.5,
    xtick distance=1,
    ytick distance=1,
    xticklabel={\color{gray}\pgfmathprintnumber\tick},
    yticklabel={\color{gray}\pgfmathprintnumber\tick},
    axis x line=middle,
    axis y line=middle,
    ]
    \draw[black,line width=1pt] (axis cs:6,0) -- (axis cs:0,6);
    \draw[blue,line width=2pt,dotted] (axis cs:0,0) -- (axis cs:6.5,6.5);

    \draw[triang] (axis cs:6,0) -- (axis cs:5,0) -- (axis cs:5,1) -- cycle;
    \draw[triang] (axis cs:5,0) -- (axis cs:5,1) -- (axis cs:4,2) -- cycle;
    \draw[triang] (axis cs:5,0) -- (axis cs:4,1) -- (axis cs:4,2) -- cycle;
    \draw[triang] (axis cs:4,2) -- (axis cs:3,2) -- (axis cs:3,3) -- cycle;
    \draw[triang] (axis cs:2,2) -- (axis cs:3,2) -- (axis cs:3,3) -- cycle;
    \draw[triang] (axis cs:1,1) -- (axis cs:2,1) -- (axis cs:2,2) -- cycle;
    \draw[triang] (axis cs:0,0) -- (axis cs:1,0) -- (axis cs:1,1) -- cycle;
    \draw[triang] (axis cs:3,0) -- (axis cs:2,1) -- (axis cs:1,1) -- cycle;
    \draw[triang] (axis cs:3,0) -- (axis cs:2,0) -- (axis cs:1,1) -- cycle;
    \draw[triang] (axis cs:2,0) -- (axis cs:1,0) -- (axis cs:1,1) -- cycle;
    \draw[triang] (axis cs:3,0) -- (axis cs:2,2) -- (axis cs:2,1) -- cycle;
    \draw[triang] (axis cs:3,0) -- (axis cs:2,2) -- (axis cs:3,1) -- cycle;
    \draw[triang] (axis cs:5,0) -- (axis cs:4,1) -- (axis cs:3,1) -- cycle;
    \draw[triang] (axis cs:5,0) -- (axis cs:4,0) -- (axis cs:3,1) -- cycle;
    \draw[triang] (axis cs:3,0) -- (axis cs:4,0) -- (axis cs:3,1) -- cycle;
    \draw[triang] (axis cs:2,2) -- (axis cs:3,2) -- (axis cs:3,1) -- cycle;
    \draw[triang] (axis cs:3,2) -- (axis cs:4,1) -- (axis cs:3,1) -- cycle;
    \draw[triang] (axis cs:3,2) -- (axis cs:4,1) -- (axis cs:4,2) -- cycle;
    
    \begin{scope}[rotate=45, yscale=-1, rotate=-45]
      \draw[triang,mirror] (axis cs:6,0) -- (axis cs:5,0) -- (axis cs:5,1) -- cycle;
      \draw[triang,mirror] (axis cs:5,0) -- (axis cs:5,1) -- (axis cs:4,2) -- cycle;
      \draw[triang,mirror] (axis cs:5,0) -- (axis cs:4,1) -- (axis cs:4,2) -- cycle;
      \draw[triang,mirror] (axis cs:4,2) -- (axis cs:3,2) -- (axis cs:3,3) -- cycle;
      \draw[triang,mirror] (axis cs:2,2) -- (axis cs:3,2) -- (axis cs:3,3) -- cycle;
      \draw[triang,mirror] (axis cs:1,1) -- (axis cs:2,1) -- (axis cs:2,2) -- cycle;
      \draw[triang,mirror] (axis cs:0,0) -- (axis cs:1,0) -- (axis cs:1,1) -- cycle;
      \draw[triang,mirror] (axis cs:3,0) -- (axis cs:2,1) -- (axis cs:1,1) -- cycle;
      \draw[triang,mirror] (axis cs:3,0) -- (axis cs:2,0) -- (axis cs:1,1) -- cycle;
      \draw[triang,mirror] (axis cs:2,0) -- (axis cs:1,0) -- (axis cs:1,1) -- cycle;
      \draw[triang,mirror] (axis cs:3,0) -- (axis cs:2,2) -- (axis cs:2,1) -- cycle;
      \draw[triang,mirror] (axis cs:3,0) -- (axis cs:2,2) -- (axis cs:3,1) -- cycle;
      \draw[triang,mirror] (axis cs:5,0) -- (axis cs:4,1) -- (axis cs:3,1) -- cycle;
      \draw[triang,mirror] (axis cs:5,0) -- (axis cs:4,0) -- (axis cs:3,1) -- cycle;
      \draw[triang,mirror] (axis cs:3,0) -- (axis cs:4,0) -- (axis cs:3,1) -- cycle;
      \draw[triang,mirror] (axis cs:2,2) -- (axis cs:3,2) -- (axis cs:3,1) -- cycle;
      \draw[triang,mirror] (axis cs:3,2) -- (axis cs:4,1) -- (axis cs:3,1) -- cycle;
      \draw[triang,mirror] (axis cs:3,2) -- (axis cs:4,1) -- (axis cs:4,2) -- cycle;
    \end{scope}

    \draw[orange,line width=2pt] (axis cs:4,3) edge[out=90,in=45,->] node[pos=.65,anchor=south] {$\pi$} (axis cs:3,4);

    \foreach \x in {6,...,0} 
    \foreach \y in {6,...,\x} {
      \pgfmathtruncatemacro{\yy}{6-\y}
      \edef\temp{\noexpand\node[roundnode] (\x-\y) at (axis cs:\yy,\x) {};}
        \temp
      }
  \end{axis}
\end{tikzpicture}
\caption{Extending a triangulation of the half-triangle \(\dHalfTriang\) to an \(\subgroup[H]\)-invariant triangulation of \(\dTriang\).}
\label{fig:flip-small-triangle}
\end{figure}

\begin{figure}[th]
  \begin{tikzpicture}[
  scale=\scalefactor,
  line width=1pt,
  roundnode/.style={circle, draw=black, fill=black, thin, inner sep=1pt},
  triang/.style={draw=blue,fill=blue,fill opacity=.25,line width=2pt,line cap=round,line join=round},
  strip/.style={draw=orange,fill=orange,fill opacity=.25,line width=2pt,line cap=round,line join=round},
  ]
  \begin{axis}[
    axis equal image,
    grid=both,
    ymin=0,ymax=7.5,xmin=0,xmax=7.5,
    xtick distance=1,
    ytick distance=1,
    xticklabel={\color{gray}\pgfmathprintnumber\tick},
    yticklabel={\color{gray}\pgfmathprintnumber\tick},
    axis x line=middle,
    axis y line=middle,
    ]

    \draw[black,line width=1pt,dotted] (axis cs:0,6) -- (axis cs:3,3);
    \draw[black,line width=1pt] (axis cs:6,0) -- (axis cs:3,3);
    \draw[blue,line width=2pt,dotted] (axis cs:0,0) -- (axis cs:7.5,7.5);

    \draw[triang] (axis cs:6,0) -- (axis cs:5,0) -- (axis cs:5,1) -- cycle;
    \draw[triang] (axis cs:5,1) -- (axis cs:4,1) -- (axis cs:4,2) -- cycle;
    \draw[triang] (axis cs:4,2) -- (axis cs:3,2) -- (axis cs:3,3) -- cycle;
    \draw[triang] (axis cs:0,0) -- (axis cs:1,0) -- (axis cs:1,1) -- cycle;
    \draw[triang] (axis cs:1,1) -- (axis cs:2,1) -- (axis cs:2,2) -- cycle;
    \draw[triang] (axis cs:2,2) -- (axis cs:3,2) -- (axis cs:3,3) -- cycle;

    \draw[strip] (axis cs:1,0) -- (axis cs:2,0) -- (axis cs:2,1) -- (axis cs:1,1) -- cycle;
    \draw[strip] (axis cs:2,0) -- (axis cs:3,0) -- (axis cs:3,2) -- (axis cs:2,2) -- cycle;
    \draw[strip] (axis cs:4,0) -- (axis cs:3,0) -- (axis cs:3,2) -- (axis cs:4,2) -- cycle;
    \draw[strip] (axis cs:5,0) -- (axis cs:4,0) -- (axis cs:4,1) -- (axis cs:5,1) -- cycle;

    \foreach \x in {6,...,0} 
    \foreach \y in {6,...,\x} {
        \pgfmathtruncatemacro{\yy}{6-\y}
        \edef\temp{\noexpand\node[roundnode] at (axis cs:\yy,\x) {};}
        \temp
      }
    \end{axis}
\end{tikzpicture}
  \qquad
  \begin{tikzpicture}[
  scale=\scalefactor,
  line width=1pt,
  roundnode/.style={circle, draw=black, fill=black, thin, inner sep=1pt},
  triang/.style={draw=blue,fill=blue,fill opacity=.25,line width=2pt,line cap=round,line join=round},
  strip/.style={draw=orange,fill=orange,fill opacity=.25,line width=2pt,line cap=round,line join=round},
  ]
  \begin{axis}[
    axis equal image,
    grid=both,
    ymin=0,ymax=7.5,xmin=0,xmax=7.5,
    xtick distance=1,
    ytick distance=1,
    xticklabel={\color{gray}\pgfmathprintnumber\tick},
    yticklabel={\color{gray}\pgfmathprintnumber\tick},
    axis x line=middle,
    axis y line=middle,
    ]

    \draw[black,line width=1pt,dotted] (axis cs:0,7) -- (axis cs:3.5,3.5);
    \draw[black,line width=1pt] (axis cs:7,0) -- (axis cs:3.5,3.5);
    \draw[blue,line width=2pt,dotted] (axis cs:0,0) -- (axis cs:7.5,7.5);

    \draw[triang] (axis cs:7,0) -- (axis cs:6,0) -- (axis cs:6,1) -- cycle;
    \draw[triang] (axis cs:6,1) -- (axis cs:5,1) -- (axis cs:5,2) -- cycle;
    \draw[triang] (axis cs:5,2) -- (axis cs:4,2) -- (axis cs:4,3) -- cycle;
    \draw[triang] (axis cs:0,0) -- (axis cs:1,0) -- (axis cs:1,1) -- cycle;
    \draw[triang] (axis cs:1,1) -- (axis cs:2,1) -- (axis cs:2,2) -- cycle;
    \draw[triang] (axis cs:2,2) -- (axis cs:3,2) -- (axis cs:3,3) -- cycle;
    \draw[triang] (axis cs:3.5,3.5) -- (axis cs:4,3) -- (axis cs:3,3) -- cycle;

    \draw[strip] (axis cs:1,0) rectangle (axis cs:2,1);
    \draw[strip] (axis cs:2,0) rectangle (axis cs:3,2);
    \draw[strip] (axis cs:3,0) rectangle (axis cs:4,3);
    \draw[strip] (axis cs:4,0) rectangle (axis cs:5,2);
    \draw[strip] (axis cs:5,0) rectangle (axis cs:6,1);

    \foreach \x in {7,...,0} 
    \foreach \y in {7,...,\x} {
        \pgfmathtruncatemacro{\yy}{7-\y}
        \edef\temp{\noexpand\node[roundnode] at (axis cs:\yy,\x) {};}
        \temp
      }
    \end{axis}
\end{tikzpicture}
  \caption{Subdividing $\dHalfTriang$ into vertical strips to obtain a lower bound of unimodular triangulations.
    When $d$ is even, $\dHalfTriang$ itself is symmetric along the $\{x=\frac{d}2\}$ line. For $d$ odd,
    there is an additional vertical strip centered at $x=\frac{d}2$.}
  \label{fig:small-triangle-strips}
\end{figure}

Kaibel and Ziegler give an argument \cite[\S2.1]{KaibelZiegler:2003} for the number of unimodular triangulations of a vertical strip of shape \(1\times n\) being  \[
  \numRectangle(1,n) = \binom{2n}{n}.
\] This provides the following first lower bound \(\LowerBound1\).
If \(d\) is even, the number of unimodular triangulations of all vertical strips together is 
\begin{equation}\label{eq:first-lower-bound}
  \LowerBound1(d) \coloneqq \prod_{n=0}^{\frac{d}2 - 1} \numRectangle(1,n)^2
  = \prod_{n=0}^{\frac{d}2 - 1} \binom{2n}{n}^2.
\end{equation}
If \(d\) is odd, the count is similar but we have an additional vertical strip sitting in between. Thus, we get \[
  \LowerBound1(d) \coloneqq \binom{d-1}{\frac{d-1}2}\prod_{n=0}^{\lfloor\frac{d}2 - 1\rfloor} \numRectangle(1,n)^2
  = \binom{d-1}{\frac{d-1}2}\prod_{n=0}^{\lfloor\frac{d}2 - 1\rfloor} \binom{2n}{n}^2
  \quad\text{for } d \text{ odd}.
\]

The following computation gives an explicit lower bound in the even case,
the odd case is carried out analogously:
\[
  \LowerBound1(d)
  = \prod_{n=0}^{\frac{d}2 - 1} \binom{2n}{n}^2 \geq \prod_{n=1}^{\frac{d}2 - 1} \Bigl( \frac1{2n}2^{2n} \Bigr)^2 = 2^{2\sum_{n=1}^{\frac{d}2-1} 2n - \log{2n}} \geq 2^{\frac{d^2}2 - d - (d-2)\log{d}}.
\]
This computation leads to the following lower bound.
\begin{lemma}
  We have \(\numSymmetricUnimodular(d) \geq \LowerBound1(d) \geq 2^{\Omega(d^2)}\).
\end{lemma}

\begin{remark}
As each stripe can vary in height, we can refine this lower bound to
\begin{equation}\label{eq:second-lower-bound}
  \LowerBound2(d) \coloneqq \begin{cases}
    \prod_{n=0}^{\frac{d}2 - 1}\Bigl(\sum_{i=0}^n \binom{2i}{i}\Bigr)^2,& d\text{ even},\\[.9em]
    \binom{d-1}{\frac{d-1}2}\prod_{n=0}^{\lfloor\frac{d}2 - 1\rfloor}\Bigl(\sum_{i=0}^n \binom{2i}{i}\Bigr)^2,& d\text{ odd}.
  \end{cases}
\end{equation}
The additional factor of \(\LowerBound2\) for odd \(d\) does \emph{not} get replaced by a sum of binomial coefficients
since the triangle on top of the central strip is fixed, due to $(\frac{d}2,\frac{d}2)$ being only a vertex of \(\dHalfTriang\)
but not a lattice point of \(\dTriang\). See Figure~\ref{fig:small-triangle-strips} for an example.
Also, \(\LowerBound2\) still only leads to a quadratic lower bound and thus gives not much improvement
over bounding by \(\LowerBound1\).
\end{remark}

\section{Upper bounds}

There are \(\subgroup[H]\)-invariant unimodular triangulations that cannot be obtained (directly) from a unimodular triangulation
of \(\dHalfTriang\) using the just described method.
Still, such a triangulation can be obtained from one as above by picking a subset of squares across the \(\symmetryAxis\) line and prescribing 
a split into triangles. This is based on the following observation.

\begin{figure}
  \begin{tikzpicture}[
  scale=\scalefactor,
  line width=1pt,
  roundnode/.style={circle, draw=black, fill=black, thin, inner sep=1pt},
  triang/.style={draw=blue,fill=blue,fill opacity=.25,line width=2pt,line cap=round,line join=round},
  ]
  \begin{axis}[
    axis equal image,
    grid=both,
    ymin=0,ymax=7.5,xmin=0,xmax=7.5,
    xtick distance=1,
    ytick distance=1,
    xticklabel={\color{gray}\pgfmathprintnumber\tick},
    yticklabel={\color{gray}\pgfmathprintnumber\tick},
    axis x line=middle,
    axis y line=middle,
    ]
    \draw[black,line width=1pt] (axis cs:6,0) -- (axis cs:0,6);
    \draw[blue,line width=2pt,dotted] (axis cs:0,0) -- (axis cs:7.5,7.5);

    \draw[triang] (axis cs:6,0) -- (axis cs:5,0) -- (axis cs:5,1) -- cycle;
    \draw[triang] (axis cs:5,0) -- (axis cs:5,1) -- (axis cs:4,2) -- cycle;
    \draw[triang] (axis cs:5,0) -- (axis cs:4,1) -- (axis cs:4,2) -- cycle;
    \draw[triang] (axis cs:4,2) -- (axis cs:3,2) -- (axis cs:3,3) -- cycle;
    \draw[triang] (axis cs:3,0) -- (axis cs:2,1) -- (axis cs:1,1) -- cycle;
    \draw[triang] (axis cs:3,0) -- (axis cs:2,0) -- (axis cs:1,1) -- cycle;
    \draw[triang] (axis cs:2,0) -- (axis cs:1,0) -- (axis cs:1,1) -- cycle;
    \draw[triang] (axis cs:3,0) -- (axis cs:2,2) -- (axis cs:2,1) -- cycle;
    \draw[triang] (axis cs:3,0) -- (axis cs:2,2) -- (axis cs:3,1) -- cycle;
    \draw[triang] (axis cs:5,0) -- (axis cs:4,1) -- (axis cs:3,1) -- cycle;
    \draw[triang] (axis cs:5,0) -- (axis cs:4,0) -- (axis cs:3,1) -- cycle;
    \draw[triang] (axis cs:3,0) -- (axis cs:4,0) -- (axis cs:3,1) -- cycle;
    \draw[triang] (axis cs:2,2) -- (axis cs:3,2) -- (axis cs:3,1) -- cycle;
    \draw[triang] (axis cs:3,2) -- (axis cs:4,1) -- (axis cs:3,1) -- cycle;
    \draw[triang] (axis cs:3,2) -- (axis cs:4,1) -- (axis cs:4,2) -- cycle;

    \begin{scope}[rotate=45, yscale=-1, rotate=-45]
      \draw[triang] (axis cs:6,0) -- (axis cs:5,0) -- (axis cs:5,1) -- cycle;
      \draw[triang] (axis cs:5,0) -- (axis cs:5,1) -- (axis cs:4,2) -- cycle;
      \draw[triang] (axis cs:5,0) -- (axis cs:4,1) -- (axis cs:4,2) -- cycle;
      \draw[triang] (axis cs:4,2) -- (axis cs:3,2) -- (axis cs:3,3) -- cycle;
      \draw[triang] (axis cs:3,0) -- (axis cs:2,1) -- (axis cs:1,1) -- cycle;
      \draw[triang] (axis cs:3,0) -- (axis cs:2,0) -- (axis cs:1,1) -- cycle;
      \draw[triang] (axis cs:2,0) -- (axis cs:1,0) -- (axis cs:1,1) -- cycle;
      \draw[triang] (axis cs:3,0) -- (axis cs:2,2) -- (axis cs:2,1) -- cycle;
      \draw[triang] (axis cs:3,0) -- (axis cs:2,2) -- (axis cs:3,1) -- cycle;
      \draw[triang] (axis cs:5,0) -- (axis cs:4,1) -- (axis cs:3,1) -- cycle;
      \draw[triang] (axis cs:5,0) -- (axis cs:4,0) -- (axis cs:3,1) -- cycle;
      \draw[triang] (axis cs:3,0) -- (axis cs:4,0) -- (axis cs:3,1) -- cycle;
      \draw[triang] (axis cs:2,2) -- (axis cs:3,2) -- (axis cs:3,1) -- cycle;
      \draw[triang] (axis cs:3,2) -- (axis cs:4,1) -- (axis cs:3,1) -- cycle;
      \draw[triang] (axis cs:3,2) -- (axis cs:4,1) -- (axis cs:4,2) -- cycle;
    \end{scope}

    \draw[triang,orange] (axis cs:0,0) rectangle (axis cs:1,1);
    \draw[triang,orange] (axis cs:1,1) rectangle (axis cs:2,2);
    \draw[triang,orange] (axis cs:2,2) rectangle (axis cs:3,3);
    \draw[line width=2pt,orange,dotted] (axis cs:0,0) -- (axis cs:3,3);
    \draw[line width=2pt,orange,dotted] (axis cs:0,1) -- (axis cs:1,0);
    \draw[line width=2pt,orange,dotted] (axis cs:1,2) -- (axis cs:2,1);
    \draw[line width=2pt,orange,dotted] (axis cs:2,3) -- (axis cs:3,2);

    \foreach \x in {6,...,0} 
    \foreach \y in {6,...,\x} {
      \pgfmathtruncatemacro{\yy}{6-\y}
      \edef\temp{\noexpand\node[roundnode] (\x-\y) at (axis cs:\yy,\x) {};}
        \temp
      }
  \end{axis}
\end{tikzpicture}
  \qquad
  \begin{tikzpicture}[
  scale=\scalefactor,
  line width=1pt,
  roundnode/.style={circle, draw=black, fill=black, thin, inner sep=1pt},
  triang/.style={draw=blue,fill=blue,fill opacity=.25,line width=2pt,line cap=round,line join=round},
  ]
  \begin{axis}[
    axis equal image,
    grid=both,
    ymin=0,ymax=7.5,xmin=0,xmax=7.5,
    xtick distance=1,
    ytick distance=1,
    xticklabel={\color{gray}\pgfmathprintnumber\tick},
    yticklabel={\color{gray}\pgfmathprintnumber\tick},
    axis x line=middle,
    axis y line=middle,
    ]
    \draw[black,line width=1pt] (axis cs:7,0) -- (axis cs:0,7);
    \draw[blue,line width=2pt,dotted] (axis cs:0,0) -- (axis cs:7.5,7.5);

    \draw[triang] (axis cs:6,0) -- (axis cs:5,0) -- (axis cs:5,1) -- cycle;
    \draw[triang] (axis cs:5,0) -- (axis cs:5,1) -- (axis cs:4,2) -- cycle;
    \draw[triang] (axis cs:5,0) -- (axis cs:4,1) -- (axis cs:4,2) -- cycle;
    \draw[triang] (axis cs:4,2) -- (axis cs:3,2) -- (axis cs:3,3) -- cycle;
    \draw[triang] (axis cs:3,0) -- (axis cs:2,1) -- (axis cs:1,1) -- cycle;
    \draw[triang] (axis cs:3,0) -- (axis cs:2,0) -- (axis cs:1,1) -- cycle;
    \draw[triang] (axis cs:2,0) -- (axis cs:1,0) -- (axis cs:1,1) -- cycle;
    \draw[triang] (axis cs:3,0) -- (axis cs:2,2) -- (axis cs:2,1) -- cycle;
    \draw[triang] (axis cs:3,0) -- (axis cs:2,2) -- (axis cs:3,1) -- cycle;
    \draw[triang] (axis cs:5,0) -- (axis cs:4,1) -- (axis cs:3,1) -- cycle;
    \draw[triang] (axis cs:5,0) -- (axis cs:4,0) -- (axis cs:3,1) -- cycle;
    \draw[triang] (axis cs:3,0) -- (axis cs:4,0) -- (axis cs:3,1) -- cycle;
    \draw[triang] (axis cs:2,2) -- (axis cs:3,2) -- (axis cs:3,1) -- cycle;
    \draw[triang] (axis cs:3,2) -- (axis cs:4,1) -- (axis cs:3,1) -- cycle;
    \draw[triang] (axis cs:3,2) -- (axis cs:4,1) -- (axis cs:4,2) -- cycle;

    \draw[triang] (axis cs:3,3) -- (axis cs:4,3) -- (axis cs:3,4) -- cycle;

    \draw[triang] (axis cs:3,3) -- (axis cs:5,2) -- (axis cs:4,2) -- cycle;
    \draw[triang] (axis cs:3,3) -- (axis cs:5,2) -- (axis cs:4,3) -- cycle;
    \draw[triang] (axis cs:5,1) -- (axis cs:5,2) -- (axis cs:4,2) -- cycle;
    \draw[triang] (axis cs:5,1) -- (axis cs:5,2) -- (axis cs:6,1) -- cycle;
    \draw[triang] (axis cs:5,1) -- (axis cs:6,0) -- (axis cs:7,0) -- cycle;
    \draw[triang] (axis cs:5,1) -- (axis cs:6,1) -- (axis cs:7,0) -- cycle;

    \begin{scope}[rotate=45, yscale=-1, rotate=-45]
      \draw[triang] (axis cs:6,0) -- (axis cs:5,0) -- (axis cs:5,1) -- cycle;
      \draw[triang] (axis cs:5,0) -- (axis cs:5,1) -- (axis cs:4,2) -- cycle;
      \draw[triang] (axis cs:5,0) -- (axis cs:4,1) -- (axis cs:4,2) -- cycle;
      \draw[triang] (axis cs:4,2) -- (axis cs:3,2) -- (axis cs:3,3) -- cycle;
      \draw[triang] (axis cs:3,0) -- (axis cs:2,1) -- (axis cs:1,1) -- cycle;
      \draw[triang] (axis cs:3,0) -- (axis cs:2,0) -- (axis cs:1,1) -- cycle;
      \draw[triang] (axis cs:2,0) -- (axis cs:1,0) -- (axis cs:1,1) -- cycle;
      \draw[triang] (axis cs:3,0) -- (axis cs:2,2) -- (axis cs:2,1) -- cycle;
      \draw[triang] (axis cs:3,0) -- (axis cs:2,2) -- (axis cs:3,1) -- cycle;
      \draw[triang] (axis cs:5,0) -- (axis cs:4,1) -- (axis cs:3,1) -- cycle;
      \draw[triang] (axis cs:5,0) -- (axis cs:4,0) -- (axis cs:3,1) -- cycle;
      \draw[triang] (axis cs:3,0) -- (axis cs:4,0) -- (axis cs:3,1) -- cycle;
      \draw[triang] (axis cs:2,2) -- (axis cs:3,2) -- (axis cs:3,1) -- cycle;
      \draw[triang] (axis cs:3,2) -- (axis cs:4,1) -- (axis cs:3,1) -- cycle;
      \draw[triang] (axis cs:3,2) -- (axis cs:4,1) -- (axis cs:4,2) -- cycle;

      \draw[triang] (axis cs:3,3) -- (axis cs:5,2) -- (axis cs:4,2) -- cycle;
      \draw[triang] (axis cs:3,3) -- (axis cs:5,2) -- (axis cs:4,3) -- cycle;
      \draw[triang] (axis cs:5,1) -- (axis cs:5,2) -- (axis cs:4,2) -- cycle;
      \draw[triang] (axis cs:5,1) -- (axis cs:5,2) -- (axis cs:6,1) -- cycle;
      \draw[triang] (axis cs:5,1) -- (axis cs:6,0) -- (axis cs:7,0) -- cycle;
      \draw[triang] (axis cs:5,1) -- (axis cs:6,1) -- (axis cs:7,0) -- cycle;
    \end{scope}

    \draw[triang,orange] (axis cs:0,0) rectangle (axis cs:1,1);
    \draw[triang,orange] (axis cs:1,1) rectangle (axis cs:2,2);
    \draw[triang,orange] (axis cs:2,2) rectangle (axis cs:3,3);
    \draw[line width=2pt,orange,dotted] (axis cs:0,0) -- (axis cs:3,3);
    \draw[line width=2pt,orange,dotted] (axis cs:0,1) -- (axis cs:1,0);
    \draw[line width=2pt,orange,dotted] (axis cs:1,2) -- (axis cs:2,1);
    \draw[line width=2pt,orange,dotted] (axis cs:2,3) -- (axis cs:3,2);

    \foreach \x in {7,...,0} 
    \foreach \y in {7,...,\x} {
      \pgfmathtruncatemacro{\yy}{7-\y}
      \edef\temp{\noexpand\node[roundnode] (\x-\y) at (axis cs:\yy,\x) {};}
        \temp
      }
  \end{axis}
\end{tikzpicture}
  \caption{Symmetric triangulations of $\dTriang$ for $d=6,7$ together with quadrangles along the $\symmetryAxis$ line.
    Note how for odd $d$, the triangle consisting of the points $(\lfloor\frac{d}2\rfloor,\lfloor\frac{d}2\rfloor)$, 
    $(\lfloor\frac{d}2\rfloor,\lceil\frac{d}2\rceil)$ and $(\lceil\frac{d}2\rceil,\lfloor\frac{d}2\rfloor)$ is fixed by the symmetry of the triangulation.}
  \label{fig:antidiagonal}
\end{figure}

\begin{lemma}\label{lem:dichotomy-symmetric-triangles}
  There are only two kinds of \(\subgroup[H]\)-feasible unimodular lattice triangles \(\triang\subset\dTriang\):
  \begin{enumerate}
    \item Either \(\triang\subset \dHalfTriang\) or \(\affineSymmetry(\triang)\subset \dHalfTriang\).
    \item There exist \(x\in\{1,\dots,\lceil\frac{d}2\rceil\}\) and \(x'\in\{x-1,x\}\) with \(2x'\leq d\) such that
      \[
        \sigma = \{(x,x-1),(x-1,x),(x',x')\}.
      \]
  \end{enumerate}
\end{lemma}
\begin{proof}
    For any \(\subgroup[H]\)-feasible unimodular lattice triangle \(\triang\subset\dTriang\) there either
    exists another triangle $\triang' = \affineSymmetry(\triang)$ which intersects properly with \(\triang\)
    or $\triang = \affineSymmetry(\triang)$.

    \begin{figure}[b]
      \begin{tikzpicture}[
  scale=.85*\scalefactor,
  line width=1pt,
  roundnode/.style={circle, draw=black, fill=black, thin, inner sep=1pt},
  triang/.style={draw=blue,fill=blue,fill opacity=.25,line width=2pt,line cap=round,line join=round},
  symaxis/.style={draw=black,line width=2pt,dotted}
  ]
  
  \draw[triang] (1,1) -- (1,2) -- (0,2) -- cycle;
  \draw[triang,orange] (2,3) -- (3,1) -- (0,4) -- cycle;
  \draw[triang,orange] (2,4) -- (3,3) -- (4,5) -- cycle;

  \begin{scope}[xshift=3cm]
    \draw[triang,orange] (4,1) -- (4,2) -- (5,0) -- cycle;
    \draw[triang,orange,draw opacity=.6,dashed] (4,1) -- (5,0) -- (6,0) -- cycle;

    \draw[triang,orange] (5,3) -- (6,1) -- (7,2) -- cycle;
    \draw[triang,orange,draw opacity=.6,dashed] (7,1) -- (5,2) -- (6,3) -- cycle;

    \draw[triang,orange] (8,4) -- (7,5) -- (8,3) -- cycle;
    \draw[triang,orange,draw opacity=.6,dashed] (7,4) -- (8,4) -- (9,3) -- cycle;
  \end{scope}

  \foreach \x in {12,...,0} 
  \foreach \y in {5,...,0} {
    \node[roundnode] at (\x,\y) {};
  }

  \foreach \x in {5,...,6} 
  \foreach \y in {5,...,0} {
    \node[roundnode,white,inner sep=2pt] at (\x,\y) {};
  }

  \node[anchor=south east] at (0,2) {a)};
  \node[anchor=south east] at (0,4) {b)};
  \node[anchor=south east] at (2,4) {c)};
  \begin{scope}[xshift=3cm]
    \node[anchor=south east] at (4,2) {d)};
    \node[anchor=south east] at (5,3) {e)};
    \node[anchor=north east] at (7,5) {f)};
    \draw[symaxis] (4,0) -- (9,5);
  \end{scope}

  \draw[symaxis] (0,1) -- (4,5);

\end{tikzpicture}
      \caption{Cases of triangles discussed in the proof of Lemma~\ref{lem:dichotomy-symmetric-triangles}.
        Cases b) and c) fail to be unimodular, while d)--f) fail to be \(\subgroup[H]\)-feasible.
      }
      \label{fig:axis-crossing}
    \end{figure}
    
    If $\triang = \affineSymmetry(\triang)$, due to symmetry \(\sigma\) must contain two vertices 
    $(x,y)$ and $(y,x)$ since $\sigma$ is a proper triangle where not all points are colinear, see Figure~\ref{fig:axis-crossing}a)--c).
    Then, we must have $y = x-1$ due to unimodularity. This fixes the third vertex
    of $\sigma$ to lie on the $\symmetryAxis$ line, and again due to unimodularity, this point is either $(x,x)$ or $(x-1,x-1)$.
    This is case a) in Figure~\ref{fig:axis-crossing} as opposed to cases b) and c).
    This also means that \(\triang = \affineSymmetry(\triang)\) implies \(\triang\not\subset\dHalfTriang\).

    Now suppose that $\triang\neq\affineSymmetry(\triang)$ but $\triang\not\subset \dHalfTriang$.
    Then \(\triang\) has vertices strictly on both sides of the $\symmetryAxis$ line, since otherwise
    \(\triang\) or \(\affineSymmetry(\triang)\) is contained in \(\dHalfTriang\).
    Choose such vertices $P_1$ and $Q$ on opposite sides with $P_1 \neq \pi(Q)$.
    This is possible: if two vertices lie strictly on the same side, at most one of them equals $\pi(Q)$;
    otherwise the third vertex lies on the $\symmetryAxis$ line, and $P_1 = \pi(Q)$ would imply $\triang = \affineSymmetry(\triang)$.
    Then, the line segments $[P_1,Q]$ and $[\pi(P_1),\pi(Q)]$ intersect transversely in the \(\symmetryAxis\) line
    since \(\pi([P_1,Q]) = [\pi(P_1),\pi(Q)]\). As this point lies in the relative interior of both line segments,
    \(\triang\) and \(\pi(\triang)\) cannot coexist in the same triangulation.
    This contradicts the assumption that \(\triang\not\subset \dHalfTriang\).
\end{proof}
In particular, this observation fixes the small triangle at the top of \(\dHalfTriang\) for odd \(d\)
in Figure~\ref{fig:odd-small-triangle}.
Using this property, knowing the number \(\numHalf(d)\) of unimodular 
triangulations for \(\dHalfTriang\) gives a bound on \(\numSymmetricUnimodular(d)\).
An example of this is shown in Figure~\ref{fig:antidiagonal} for $d=6$ and $d=7$.
\begin{lemma}
  The number of symmetric unimodular triangulations of \(\dTriang\) is bounded by \[
    \numHalf(d)
    \leq
    \numSymmetricUnimodular(d)
    \leq
    2^{\lfloor\frac{d}2\rfloor}\cdot\numHalf(d).
  \]
\end{lemma}
\begin{proof}
    Due to Lemma~\ref{lem:dichotomy-symmetric-triangles}, we obtain a \(\subgroup[H]\)-invariant triangulation of \(\dTriang\)
    by prescribing a split for each square along the \(\symmetryAxis\) line as in Figure~\ref{fig:antidiagonal}
    and choosing a unimodular triangulation for the remainder of \(\dHalfTriang\). There are \(\lfloor\frac{d}2\rfloor\) squares along the
    \(\symmetryAxis\) line leading to \(2^{\lfloor\frac{d}2\rfloor}\) choices for splits.

    Splitting every square along the \(\symmetryAxis\) line leads to the lower bound as in the discussion leading to \eqref{eq:half-lower-bound}.
    Since extending a unimodular triangulation \(\triangulation\) of \(\dHalfTriang\) to a \(\subgroup[H]\)-invariant triangulation of \(\dTriang\) introduces at
    most \(\lfloor\frac{d}2\rfloor\) squares along the symmetry axis, we can obtain at most \(2^{\lfloor\frac{d}2\rfloor}\) additional 
    triangulations from \(\triangulation\) by splitting those squares differently.
    This leads to the upper bound.
\end{proof}

These observations allow us to obtain upper bounds for \(\numSymmetricUnimodular(d)\) from those for \(\numHalf(d)\).
From our lower bounds we obtained that \(\logSymmetricUnimodular(d)\)
grows at least quadratically. By comparing with results on the number of unimodular triangulations
of an \(n\times m\)-rectangle we can already get a rough estimate for the order of growth for \(\numSymmetricUnimodular(d)\).

For the rectangle, the number of triangulations has been bounded from above by \(2^{6mn}\) due to Orevkov~\cite{Orevkov:1998}.
This was improved to an upper bound of \(2^{3mn-m-n}\) by Anclin~\cite{Anclin:2003}. 
Both \(\dHalfTriang\) for even \(d\) and \(\dHalfTriang'\) for odd \(d\) lie in the lattice rectangle \([0,d]\times[0,\lceil\frac{d}2\rceil]\),
and every unimodular triangulation of them extends to one of the rectangle, since the remaining region is a union of lattice polygons.
Setting \(n = d\) and \(m = \lceil\frac{d}2\rceil\), this gives an overestimate for \(\numHalf(d)\) by
\begin{equation}
  \numHalf(d) \leq 2^{3d\lceil\frac{d}2\rceil - d - \lceil\frac{d}2\rceil} \leq 2^{\bigO(d^2)}.
\end{equation}
This already implies that \(\logSymmetricUnimodular(d)\) grows quadratically since
\[
  \Omega(d^2)
  \leq \logHalf(d)
  \leq \logSymmetricUnimodular(d)
  \leq \logHalf(d) + \Bigl\lfloor\frac{d}2\Bigr\rfloor
  \leq \bigO(d^2) + \Bigl\lfloor\frac{d}2\Bigr\rfloor.
\]

The upper bound by Anclin holds in greater generality than for \(n\times m\)-rectangles.
We use the following result to obtain a more detailed picture on upper bounds for \(\numSymmetricUnimodular(d)\).
\begin{lemma}[{\cite[Theorem 1]{Anclin:2003}}]\label{lem:anclin-bound}
  The number of unimodular triangulations of a lattice polygon $P$ is bounded by \(2^{\numInteriorEdges}\),
  where $\numInteriorEdges$ is the number of non-boundary edges of an arbitrary unimodular triangulation of $P$.
\end{lemma}
This result gives an explicit formula for an upper bound if we obtain the number \(\numInteriorEdges(d)\)
of interior edges of a unimodular triangulation of $\dHalfTriang$ (of $\dHalfTriang'$ for odd $d$).
Using the following result, which is based on Pick's formula, this reduces to a count of interior and boundary lattice points.
\begin{lemma}[{\cite[Lemma 3.1.3]{LoeraRambauSantos_TriangulationsStructuresApplications_2010}}]\label{lem:pick}
  Let $\Sigma$ be a unimodular triangulation of a point set $A\subset\RR^2$.
  Let $n$ be the number of points in $A$ and $n_b$ be the number of points on the boundary of $\conv{A}$.
  Then, the number of edges in $\Sigma$ is \[
    3n - n_b - 3.
  \]
\end{lemma}
In the following, we provide explicit formulas for \(n\) and \(n_b\) in our case to obtain an explicit formula from Lemma~\ref{lem:pick}.
\begin{lemma}
  \begin{enumerate}
    \item The number of lattice points of $\dHalfTriang$ is given by \[
        n(d) = \begin{cases}
          \frac14d^2 + d + 1,& d\text{ even}, \\
          \frac14d^2 + d + \frac34,& d\text{ odd}.
        \end{cases}
      \]
    \item The number of boundary lattice points of $\dHalfTriang$ is given by \[
        n_b(d) = 2d.
      \]
  \end{enumerate}
\end{lemma}
\begin{proof}
    \begin{enumerate}
      \item In both cases, the number of lattice points of $\dHalfTriang$ satisfies the recursion \[
          n(d) = n(d-2) + d+1,
        \]
        which can be seen by separating the lattice points into the bottom-most row, which contains
        $d+1$ points, and the triangle $(d-2)\cdot T$ above.

        Since $0\cdot T$ is a single point, $n(0) = 1$ and \[
          n(d) = \sum_{k=0}^{\frac{d}2} 2k+1 = \frac{d}2\cdot\frac{d+2}2 + \frac{d}2 + 1 = \frac14d^2 + d + 1.
        \]
        On the other hand, $1\cdot T$ contains $2$ lattice points, which means that for odd $d$ we have \[
          n(d) = \sum_{k=1}^{\frac{d+1}2} 2k = \frac{d+1}2 \cdot \frac{d+3}2 = \frac14d^2 + d + \frac34.
        \]
      \item The base of $\dHalfTriang$ contains $d+1$ lattice points, while the legs each contain 
        $\lfloor\frac{d}2\rfloor$ lattice points \emph{not} on the base. If $d$ is even, the
        intersection point of both legs is a lattice point. Thus, for $d$ even, we
        have \[
          n_b(d) = d+1 + 2\frac{d}2 - 1 = 2d.
        \] For $d$ odd, the top of $\dHalfTriang$ is not a lattice point, so each leg contributes $\frac{d-1}2$ lattice points
        and we also have \[
          n_b(d) = d+1 + 2\cdot\frac{d-1}2 = 2d. \qedhere
        \]
    \end{enumerate}
\end{proof}

By Lemma~\ref{lem:pick}, a unimodular triangulation has \(3n - n_b - 3\) edges, \(n_b\) of which lie on the boundary,
so \(\numInteriorEdges(d) = 3n(d) - 2n_b(d) - 3\).
Consequently, with Lemma~\ref{lem:anclin-bound}, the following function provides an upper bound
of \(\numHalf(d)\) for \(d\geq2\):
\begin{equation}\label{eq:upper}
  \UpperBound(d) \coloneqq \begin{cases}
    2^{\frac34d^2 - d},& d\text{ even}, \\
    2^{\frac34d^2 - d - \frac34},& d\text{ odd}.
  \end{cases}
\end{equation}

\begin{theorem}\label{thm:upper-lower-bounds}
  The number \(\numHalf(d)\) of unimodular triangulations of $\dHalfTriang$
  is bounded from above by \[
    \numHalf(d) \leq \UpperBound(d) \leq 2^{\frac34d^2 - d}.
  \] The number \(\numSymmetricUnimodular(d)\) of \(\subgroup[H]\)-invariant unimodular triangulations of \(\dTriang\)
  is bounded by \[
    2^{\Omega(d^2)} \leq \numSymmetricUnimodular(d) \leq 2^{\lfloor\frac{d}2\rfloor}\cdot\UpperBound(d) \leq 2^{\bigO(d^2)}.
  \]
\end{theorem}

Kaibel and Ziegler also considered the number of triangulations 
of point configurations on a logarithmic scale.
Normalizing this by the number of points leads to the notion of \emph{capacity}
of a point configuration.

One can also ask for a similar quantity in the case of \(\subgroup[H]\)-invariant triangulations of \(\dTriang\), 
for which the (symmetric) capacity is given by \[
  \symmetricCapacity(d) \coloneqq \frac{2\logSymmetricUnimodular(d)}{d(d-1)}.
\]

Computing explicit formulas for the upper and lower bounds from Theorem~\ref{thm:upper-lower-bounds}
and taking the logarithms thereof yields
\begin{equation}\label{eqn:explicit-bounds}
  \frac12 d^2 - d - (d-2)\log{d} \leq \logSymmetricUnimodular(d) \leq \frac34d^2-\frac12d.
\end{equation}
For the asymptotic behavior of \(\symmetricCapacity(d)\), that is \(\symmetricCapacity\coloneqq\lim_{d\to\infty}\symmetricCapacity(d)\),
above inequalities suggest that \(1 \leq \symmetricCapacity \leq \frac32\).
Performing a quadratic regression on the values of \(\logSymmetricUnimodular(d)\) from Table~\ref{tbl:logarithmic-counts}
yields an approximate relation of 
\[
  \logSymmetricUnimodular(d) \approx 0.56 d^2 - 0.77 d + 0.21.
\]
Note that this regression disregards the summand of order \(d\log{d}\) from the lower bound in \eqref{eqn:explicit-bounds}. This is negligible
as the lower bound grows asymptotically at least quadratically.
This experimental observation leads to the following conjecture.
\begin{conjecture}
  $\symmetricCapacity = 1$.
\end{conjecture}

\begin{table}
  \centering
  \caption{Logarithmic number of \(\subgroup[H]\)-invariant triangulations of $\dTriang$ up to 
    \(\subgroup[H]\)-feasible symmetries in comparison with some 
    lower and upper bounds. Lower bounds have been rounded up while upper bounds have been rounded down. Here \(\logHalf(d) = \log\numHalf(d)\) with \(\numHalf(d)\) as in \eqref{eq:def-half}.}
  \label{tbl:logarithmic-counts}
  \footnotesize       
  \begin{tabular}{rrrrrrrrrr}
    \toprule
    d & 1 & 2 & 3 & 4 & 5 & 6 & 7 & 8 & 9 \\
    \midrule
    $\lowerBound2(d)$ & 0 & 0 & 1 & 3 & 6 & 10 & 14 & 20 & 26 \\
    $\logHalf(d)$ & 0 & 0 & 2 & 5 & 9 & 14 & 21 & 28 & 37 \\
    \midrule
    $\logSymmetricUnimodular(d)$ & 0 & 1 & 2.8 & 6.2 & 10.2 & 15.9 & 22.2 & 30.2 & 38.6 \\
    \midrule
    $\lfloor\frac{d}2\rfloor+\logHalf(d)$ & 0 & 1 & 3 & 6.5 & 10.8 & 16.7 & 23.2 & 31.4 & 40.1 \\
    $\upperBound(d)$ & 0 & 1 & 3 & 8 & 13 & 21 & 29 & 40 & 51 \\
    \bottomrule
\end{tabular}
\end{table}

\section*{Acknowledgments}
We are indebted to Christoph Spiegel, who confirmed the values of $\numSymmetricUnimodular(d)$ for $d\leq 8$ in Table~\ref{tbl:the-numbers} with an independent implementation tailored to the specific case.
Further, we thank an anonymous referee, whose comments helped improve the exposition.

KF is supported by the UK Research and Innovation: Engineering and Physical Sciences Research Council under grant reference [EP/Y028872/1].
MJ is supported by the Deutsche Forschungsgemeinschaft (DFG, German Research Foundation) under Germany's Excellence Strategy -- ``The Berlin Mathematics Research Center MATH$^+$'' (EXC-2046/1, EXC-2046/2, pro\-ject ID 390685689), ``Symbolic Tools in Mathematics and their Application'' (TRR 195, project ID 286237555), and ``Mathematical Research Data Initiative (MaRDI)'' (NFDI 29/2, project ID 460135501).

\printbibliography

\end{document}